\documentclass[11pt,reqno]{amsart}

\usepackage{amsmath, amsfonts, amsthm, amssymb}

\newtheorem{Theorem}{Theorem}[section]
\newtheorem{Lemma}{Lemma}[section]
\newtheorem{Proposition}{Proposition}[section]
\newtheorem{Corollary}{Corollary}[section]

\theoremstyle{definition}
\newtheorem{Definition}{Definition}[section]

\theoremstyle{remark}
\newtheorem{Remark}{Remark}[section]

\numberwithin{equation}{section}

\renewcommand{\u}{{\bf u}}
\renewcommand{\H}{{\bf H}}
\newcommand{\R}{{\mathbb R}}
\newcommand{\Dv}{{\rm div}}

\newcommand{\tr}{{\rm tr}}

\def\f{\frac}

\def\D{\Delta }

\def\hf1{^\f{1}{1-\xi^2}}

\def\be{\begin{equation}}
\def\en{\end{equation}}
\def\bs{\begin{split}}
\def\es{\end{split}}

\newcommand{\F}{{\mathtt F}}

\begin{document}

\author{Xianpeng Hu and Fang-hua Lin}
\address{Courant Institute of Mathematical sciences, New York
University, New York, NY 10012, USA.}
\email{xianpeng@cims.nyu.edu}
\address{Courant Institute of Mathematical sciences, New York
University, New York, NY 10012, USA.}
\email{linf@cims.nyu.edu}

\title[Incompressible Magnetohydrodynamics]
{Global Existence for Two Dimensional Incompressible
Magnetohydrodynamic Flows with Zero Magnetic Diffusivity}

\keywords{Incompressible MHD, global existence,
two dimensions.}

\date{\today}
\thanks{Supported by NSF Grants.}

\begin{abstract}
The existence of global-in-time classical solutions to the Cauchy
problem of incompressible Magnetohydrodynamic flows with zero
magnetic diffusivity is considered in two dimensions. The linearization of equations is a degenerated parabolic-hyperbolic system. The solution is constructed as
a small perturbation of a constant background in critical spaces. The deformation
gradient has been introduced to decouple the subtle coupling
between the flow and the magnetic field. The $L^1$ dissipation of the velocity is obtained.
\end{abstract}

\maketitle

\section{Introduction}

Magnetohydrodynamics (MHD) studies the dynamics of electically conducting fluids. Examples of such fluids include plasmas, liquid metals, and salt water or electrolytes. The fundamental concept
behind MHD is that magnetic fields can include currents in moving conductive fluid, which in turn creates forces on the fluid and also changes the magnetic field itself. In applications,
the dynamic motion of the fluid and the magnetic
field interact strongly on each other, especially when the magnetic diffusivity is small. The hydrodynamic and
electrodynamic effects are coupled. The equations describing
two-dimensional incompressible magnetohydrodynamic flows with zero magnetic diffusivity have the following
form (\cite{Ca, KL, CD, KL1, LL, MR}):
\begin{equation} \label{e1}
\begin{cases}
\partial_t\u+\u\cdot\nabla\u-{\bf B}\cdot\nabla {\bf B}+\nabla
\left(P+\f{1}{2}|{\bf B}|^2\right)=\mu \Delta \u, \\
\partial_t {\bf B}-\nabla\times(\u\times {\bf B})=0,\quad
\Dv\u=\Dv {\bf B}=0,
\end{cases}
\end{equation}
where $\u\in\R^2$ is the velocity, ${\bf B}\in\R^2$ is the magnetic field, and $P$ is the pressure of the flow. The viscosity coefficients of
the flow are independent of the magnitude and
direction of the magnetic field and satisfy $\mu>0$ which can be deduced directly
from the second law of thermodynamics, and for now on $\mu=1$ is assumed for simplicity of presentation. Usually, we refer to the second equation
in \eqref{e1} as the induction equation, and the first equation
as the momentum balance equation. It is well-known that the
electromagnetic fields are governed by the Maxwell's equations. In
magnetohydrodynamics, the displacement current  can be neglected
(\cite{KL,LL}). As a consequence, the last equation in \eqref{e1}
is called the induction equation, and the electric field can be
written in terms of the magnetic field ${\bf B}$ and the
velocity $\u$,
\begin{equation*}
\mathfrak{E}=-\u\times {\bf B}.
\end{equation*}
Although the electric field $\mathfrak{E}$ does not appear in
\eqref{e1}, it is indeed induced according to the above relation
by the moving conductive flow in the magnetic field.

The stability of \eqref{e1} is expected physically and was observed numerically, but the rigorous mathematical verification is surprisely open since the pioneering work of Hannes Alfv$\acute{e}$n, a Nobel laureate, in 1940s.
In this paper, we are interested in global classical solutions $(\u, {\bf B})$ to \eqref{e1} which are small perturbations around an equilibrium $(0,h_0)$, where, up to a scaling and a rotation of Eulerian coordinates,
the constant vector $h_0$ is assumed to be $(1,0)^\top$ in two dimensional space $\R^2$ ($v^\top$ means the
transpose of $v$). More precisely, we define
$${\bf B}=h_0+\H=(1+\H_1, \H_2)^\top,$$ and consider the global classical solutions of
\begin{equation} \label{e2}
\begin{cases}
\partial_t\u+\u\cdot\nabla\u-\H\cdot\nabla \H-h_0\cdot\nabla\H+\nabla
\left(P+\f{1}{2}|{\bf B}|^2\right)= \Delta \u, \\
\partial_t {\bf H}-\nabla\times(\u\times h_0)=\nabla\times(\u\times {\bf H}),\quad
\Dv\u=\Dv \H=0,
\end{cases}
\end{equation}
associated with the initial condition:
\begin{equation}\label{IC1}
(\u, \H)|_{t=0}=(\u_0(x), \H_0(x)), \quad x\in\R^2.
\end{equation}
One of advantages of incorporating the magnetic pressure
$|{\bf B}|^2$ into the total pressure $P_{total}=P+\f12|{\bf B}|^2$ lies in the fact that the total pressure is indeed a quadratic term in terms of perturbations and hence it is harmless
for our analysis.

The global wellposedness of \eqref{e1} is a widely open problem due to
the strong coupling between the fluid and the magnetic field. Even though the system is in the regime of incompressible flows, 
due to the coupling between the fluid and the magnetic field, the system \eqref{e1} shares a similar linearized structure as compressible Navier-Stokes equations and viscoelastic fluids; 
and hence a global existence of classical solutions with small data as \cite{RD, HW, QZ} (and references therein) is expected. 
Indeed, the main difficulty in solving \eqref{e2}-\eqref{IC1} lies in the dissipation mechanisms of the velocity and the magnetic field, which are sharply 
different from the situation for the compressible Navier-Stokes equation.   
While the dissipation for one partial derivative of the magnetic field, $\partial_{x_1}\H$, is relatively clear, the dissipation for the other derivative $\partial_{x_2}\H$ of 
the magnetic field seems complicated and subtle (see \cite{LP}).
This intrigues us to carefully analyze the linear structure of \eqref{e2}
\begin{subequations}\label{e1a}
 \begin{align}
&\partial_t\u-\D\u-\partial_{x_1}\H=0\\
&\partial_t\H-\partial_{x_1}\u=0.
 \end{align}
\end{subequations}
Indeed, taking one more derivative, one gets
$$\partial_{tt}\u-\D\partial_t\u-\partial_{x_1}^2\u=0,$$
which is a degenerated parabolic-hyperbolic system. This differential structure shows we need to control the competition between the parabolicity and the hyperbolicity.

One of the main novelties of this work is the dissipation mechanism of
the magnetic field $\H$. The dissipation of the magnetic field is intrinsically
related to the flow. This motivates us to introduce the concept of the deformation gradient (see \cite{CD,HW, LLZ1, LLZ,LZ,QZ, ST, ST1}),
which is defined to be the gradient of the flow map with respect to the Langrangian configuration. The key observation
here is that the magnetic field is closely related to the inverse of the deformation
gradient $\F$ (see Proposition \ref{p11}). The relation between the
deformation gradient and the magnetic field can be interpreted as
a ``frozen'' law in MHD. It turns out that some combination
between the deformation gradient and the magnetic field is
transported by the flow (see Proposition \ref{p11}). A direct
consequence of this ``frozen'' law and properties of the
deformation gradient is that we can decouple the relation between
the flow and the magnetic field (see Section 2 below for details). From this viewpoint, the deformation
gradient serves like a bridge to connect the flow and the magnetic
field.

Since the deformation gradient does not explicitly appear in the system of \eqref{e1}, one still needs to control the $L^\infty$ norm of the deformation gradient in time. 
This intrigues another difficulty since the deformation gradient only satisfy a tranport equation from its definition. To overcome this difficulty,
one way is to contorl $L^\infty$ norm of $\nabla\u$. For Cauchy problems, the necessarity of $L^\infty$ norms for $\nabla\u$ motivates us to work on Besov spaces of functions, since $\dot{B}_{2,1}^1(\R^2)\subset L^\infty.$  
Motivated by the work \cite{RD, HW, QZ}, it is natural to decompose the phase space for the magnetic field as suggested by the spectral analysis in Section 4. 
This consideration naturally involves the so-called \textit{hybrid Besov space} $\tilde{B}^{s,t}$ (see Definition in Section 3).  
In summary, if one is interested in the global existence of \eqref{e1} with as low as possible regularity, it seems necessary to work in the framework of Besov spaces of functions because:
\begin{itemize}
\item The system \eqref{e1} is scaling invariant in Besov spaces $(\u,{\bf B})\in \dot{B}^0_{2,1}\times\dot{B}_{2,1}^{1}$;
\item Due to the weak dissipation mechanism for the magnetic field in the direction parallel to the background, 
the $L^1$ integrability of $\|\nabla\u\|_{L^\infty}$ is necessary and which is generally obtained by using Besov spaces for Cauchy problems.
\end{itemize}

For the incompressible version of \eqref{e1}, authors in \cite{XZ, XZ1}
proposed a global existence near the equilibrium $(\u,\H)=(0,h_0)$
in Lagrangian coordinates using techniques from anisotropic Besov spaces with negative indicies. As
the viscosity $\mu$ vanishes further, system
\eqref{e1} becomes two dimensional ideal incompressible MHD.
Authors in \cite{MZ, ZF} showed a local existence for the ideal
incompressible MHD (see also \cite{CG, Yu} for problems of
current-vortex sheets). As far as the criteria for blow-up in
finite time for system \eqref{e1} is concerned, we refer the
interested reader to \cite{LMZ}. For the MHD with partial
dissipation for the magnetic field, authors in \cite{CW} showed a
global existence in two dimension without the smallness
assumption. 

This paper is organized as follows. In Section 2, we will state
our main result and explain the strategy of proof, including the
dissipation mechanism for perturbations. In Section 3, Besov spaces of funcitions, including the hybird Besov spaces will be defined and some properties of those
spaces will be listed. Section 4 is devoted
to dissipation estimates for the velocity $\u$ and $\H$, while we will
finish the proof of our main result in Section 5. 


\bigskip

\section{Main Results}

In this section, we state our main result and explain the strategy
of proof. For this purpose, we first take a look at the linearized
structure of the momentum equation in \eqref{e1}. Indeed, using
the continuity equation, the second equation in \eqref{e2} can be
rewritten as
\begin{equation}\label{11}
 \partial_t\u-h_0\cdot\nabla\H-\mu\D\u+\nabla\left(P+\f12|{\bf B}|^2\right)=-\u\cdot\nabla\u+\H\cdot\nabla\H.
\end{equation}
Define
$$\Lambda^s=\mathcal{F}^{-1}(|\xi|^s\mathcal{F}(f)),$$
where $\mathcal{F}$ denotes the Fourier transformation, and
denote
$$\omega=\Lambda^{-1}\textrm{curl}\u$$
with $\textrm{curl}\u=\partial_{x_2}\u_1-\partial_{x_1}\u_2$. 
Applying the operator $\Lambda^{-1}\textrm{curl}$ to \eqref{11} yields
\begin{equation}\label{13}
\partial_t\omega-\mu\D\omega-\Lambda\H_2=\Lambda^{-1}\textrm{curl}(\H\cdot\nabla\H-\u\cdot\nabla\u).
\end{equation}
Here for \eqref{13}, we used $\Dv\H=0$ and
\begin{equation*}
 \begin{split}
\Lambda^{-1}\textrm{curl}(h_0\cdot\nabla\H)&=\Lambda^{-1}\left(\f{\partial^2 \H_1}{\partial x_1x_2}-\f{\partial^2\H_2}{\partial x_1^2}\right)\\
&=-\Lambda^{-1}\left(\f{\partial^2\H_2}{\partial x_1^2}+\f{\partial^2 \H_2}{\partial x_2^2}\right)\\
&=\Lambda\H_2.
 \end{split}
\end{equation*}
Clearly \eqref{13} implies the dissipation of $\H_2$ or the dissipation of $\partial_{x_1}\H$ due to the incompressibility $\Dv\H=0$. 

\subsection{Dissipation of $\H$}
Let us define the flow map $x(t,\alpha)$
associated to the velocity $\u$ as
\begin{equation}\label{lang}
\f{dx(t,\alpha)}{dt}=\u(t, x(t,\alpha))\quad
\textrm{with}\quad x(0)=\alpha.
\end{equation}
We introduce the deformation gradient $\F\in M^{2\times 2}$
($M^{2\times 2}$ denotes the set of all $2\times 2$ matrices with
positive determinants) as (see \cite{CD, LLZ1, LLZ, LZ} and references therein)
$$\F(t,x(t,\alpha))=\f{\partial x(t,\alpha)}{\partial\alpha}.$$
From the chain rule, it follows that $\F$ satisfies a transport
equation in the Eulerian coordinate
\begin{equation}\label{f}
\partial_t\F+\u\cdot\nabla\F=\nabla\u\F.
\end{equation}
Denote by $A$ and $J$ the inverse and the determinant of $\F=\nabla_\alpha x$ respectively; that is
$$A=\F^{-1}\quad\textrm{and}\quad J=\det\F.$$
Since $A\F=I$, differentiating $A$ gives
\begin{equation}\label{17a}
 \f{D}{Dt} A=-A\nabla_\alpha \u A\quad\textrm{and}\quad \partial_{\alpha_i}A=-A\nabla_\alpha\partial_{\alpha_i}x A,
\end{equation}
where $\f{D}{Dt}$ stands for the material derivative.
Differentiating $J$ gives
\begin{equation}\label{17b}
 \f{D}{Dt}J=J\tr(A\nabla_\alpha\u)\quad\textrm{and}\quad \partial_{\alpha_i}J=J\tr(A\nabla_\alpha\partial_{\alpha_i}x).
\end{equation}

The divergence free condition $\Dv\u=0$, together with \eqref{17b}, yields
\begin{equation*}
J=J_0.
\end{equation*}
For simplicity of the presentation, we assume from now on that
\begin{equation}\label{19}
J_0=1, \quad\textrm{and hence} \quad J=1 \quad\textrm{for all time}. 
\end{equation}

The magnetic field is incorporated into the flow through the deformation gradient as follows.
\begin{Proposition}\label{p11}
Assume that $(\u, \F)$ satisfies the system \eqref{e1} and the equation \eqref{f}. Then one has the relation
\begin{equation}\label{bf}
A{\bf B}(t)=A_0{\bf B_0}\quad \textrm{for all}\quad t\ge 0.
\end{equation}
\end{Proposition}
\begin{proof}
In Eulerian coordinates, \eqref{17a} is interpreted as
\begin{equation}\label{17d}
 \partial_t A+\u\cdot\nabla A+A\nabla\u=0.
\end{equation}
On the other hand, using $\Dv\u=\Dv\H=0$, the second equation of \eqref{e1} can be rewritten as
\begin{equation}\label{17e}
 \partial_t{\bf B}+\u\cdot\nabla{\bf B}={\bf B}\cdot\nabla\u,
\end{equation}
which is refered as a ``frozen'' law of MHD in literatures (for example \cite{Ca,KL}).
Therefore one deduces from \eqref{17d} and \eqref{17e} that
$$\partial_t\Big(A{\bf B}\Big)+\u\cdot\nabla\Big(A{\bf B}\Big)=0,$$
and the desired identity \eqref{bf} follows.
\end{proof}

A direct consequence of Proposition \ref{p11} is that along the flow map, the quantity $A{\bf B}$ is a constant. From now on, we assume that
\begin{equation}\label{A1}
A_0{\bf B_0}=h_0.
\end{equation}
Therefore, it follows from Proposition \ref{p11} that for all $t>0$,
\begin{equation}\label{A11}
 A{\bf B}(t)=h_0.
\end{equation}
Keeping in mind that $A\F=I$ and \eqref{19}, one has
$$A_{ij}=\f{\partial \alpha_i}{\partial x_j}\quad \textrm{and}\quad \F=A^{-1}=\left[\begin{array}{ccc} A_{22},\quad -A_{12}\\ -A_{21},\quad A_{11}\end{array}\right]$$
Multiplying the identity \eqref{A11} by $\F$ implies
\begin{equation}\label{16}
\begin{split}
{\bf{B}}(t)=\F(x,t)h_0=\left(\begin{array}{ccc}A_{22}\\-A_{21}\end{array}\right).
\end{split}
\end{equation}
Introduce the perturbation of $A$ as
$$\mathcal{A}=A-I.$$
Components of \eqref{16} gives
\begin{equation}\label{16a}
\H_1=\mathcal{A}_{22}\quad\textrm{and}\quad \H_2=-\mathcal{A}_{21}. 
\end{equation}

Let us now explain the idea to obtain the dissipation for the magnetic field. Indeed the incompressibility $\Dv\H=0$ and the disspation of $\H_2$ imply the dissipation for $\partial_{x_1}\H$. On the other hand,
it holds, using \eqref{16a} and integration by parts
\begin{equation}\label{17}
 \begin{split}
(\partial_{x_1}\H|\Dv \mathcal{A})&=\sum_{j=1}^2\left(\f{\partial^2\alpha_2}{\partial x_1\partial x_2}|\f{\partial^2 \alpha_1}{\partial x_j^2}\right)
-\sum_{j=1}^2\left(\f{\partial^2\alpha_2}{\partial x_1^2}|\f{\partial^2 \alpha_2}{\partial x_j^2}\right)\\
&=-\sum_{j=1}^2\left\|\f{\partial^2\alpha_2}{\partial x_2\partial x_j}\right\|_{L^2}^2-\sum_{j=1}^2\left\|\f{\partial^2\alpha_2}{\partial x_1\partial_{x_j}}\right\|_{L^2}^2
+\sum_{i,j=1}^2\left(\f{\partial^2\alpha_i}{\partial x_i\partial x_j}|\f{\partial^2 \alpha_2}{\partial x_2\partial x_j}\right)\\
&=-\|\nabla\H\|_{L^2}^2+\sum_{i,j=1}^2\left(\f{\partial^2\alpha_i}{\partial x_i\partial x_j}|\f{\partial^2 \alpha_2}{\partial x_2\partial x_j}\right).
 \end{split}
\end{equation}
To control the last term above, we note that the identity \eqref{19} implies $1=\det A$ and hence
$$0=\tr\mathcal{A}+\det\mathcal{A}\quad\textrm{or equivalently}\quad \tr\mathcal{A}=-\det\mathcal{A}.$$
Thus one has
\begin{equation*}
 \begin{split}
\sum_{i,j=1}^2\left(\f{\partial^2\alpha_i}{\partial x_i\partial x_j}|\f{\partial^2 \alpha_2}{\partial x_2\partial x_j}\right)=
-\sum_{j=1}^2\left(\f{\partial\det\mathcal{A}}{\partial x_j}|\f{\partial^2 \alpha_2}{\partial x_2\partial x_j}\right).
 \end{split}
\end{equation*}
Thus one deduces from \eqref{17} that
\begin{equation}\label{18}
 \begin{split}
\|\nabla\H\|_{L^2}^2+\left\|\nabla\H_1\right\|_{L^2}^2&=-(\partial_{x_1}\H|\Dv \mathcal{A})-\sum_{j=1}^2\left(\f{\partial\det\mathcal{A}}{\partial x_j}|\f{\partial^2 \alpha_2}{\partial x_2\partial x_j}\right).
 \end{split}
\end{equation}
All terms in the right hand side contain at least one term with $L^1$ dissipation in time, provided that one can obtain the dissipation for $\det\mathcal{A}$. This in turns gives a $L^2$ dissipation in time for $\H$, and hence
a $L^1$ dissipation for nonlinear terms such as $|\H_1|^2$.

\begin{Remark} 
The reason why we prefer to using $A$ instead of the deformation gradient $\F$ itself lies in the fact that $A$ is actually a gradient in Eulerian coordinates, and hence taking spatial derivatives of $A$ will not involve 
change of variables.  
\end{Remark}

\subsection{Main Results}

Let us now introduce the functional space which appears in the
global existence theorem.

\begin{Definition}
For $T>0$, and $s\in \R$, we denote
\begin{equation*}
\begin{split}
\mathfrak{B}^s_T&=\Big\{(f, g)\in  \left( L^1(0,T; \hat{B}^{s+1})\cap
C([0,T];\hat{B}^{s-1})\right)\\&\qquad\qquad\qquad\times \left(L^2(0,T;
\hat{B}^{s})\cap C([0,T];\tilde{B}^{s-1,s})\right)
\end{split}
\end{equation*}
and
\begin{equation*}
\begin{split}
\|(f,g)\|_{\mathfrak{B}^s_T}&=\|f\|_{L^\infty_T(\hat{B}^{s-1})}+\|g\|_{L^\infty_T(\tilde{B}^{s-1,s})}+\|f\|_{L^1_T(\hat{B}^{s+1})}+\|g\|_{L^2_T(\hat{B}^{s})},
\end{split}
\end{equation*}
where the Besov space $\hat{B}^s$ and the \textit{hybrid Besov space} $\tilde{B}^{s,t}$ will be defined in Section 3.
We use the notation $\mathfrak{B}^s$ if $T=+\infty$ by changing
the interval $[0,T]$ into $[0,\infty)$ in the definition above.
\end{Definition}

Now we are ready to state our main theorem.
\begin{Theorem}\label{MT}
There exist two positive constants $\gamma$ and $\Gamma$,
such that, if $\u_0\in
\hat{B}^{0}$, ${\bf B}_0-h_0\in\tilde{B}^{0,1}$ and $\mathcal{A}_0\in\hat{B}^{1}$ satisfy \eqref{A1} and
$$\|\u_0\|_{\hat{B}^{0}}+\|{\bf B}-h_0\|_{\tilde{B}^{0,1}}+\|\mathcal{A}_0\|_{\hat{B}^{1}}\le
\gamma$$ for a sufficiently small $\gamma$,
then Cauchy problem \eqref{e1} with initial data \eqref{IC1} has a unique global solution $(\u, {\bf B})$ with $(\u,\H)\in \mathfrak{B}^{1}$.
Moreover, the solution satisfies the following estimate
\begin{equation}\label{ES}
\|(\u,\H)\|_{\mathfrak{B}^{1}}+\|\mathcal{A}\|_{L^\infty(\hat{B}^{1})}\le
\Gamma\left(\|\u_0\|_{\hat{B}^{0}}+\|{\bf B}-h_0\|_{\tilde{B}^{0,1}}+\|\mathcal{A}_0\|_{\hat{B}^{1}}\right).
\end{equation}
\end{Theorem}

Theorem \ref{MT} is verified by extending a local solution (see \cite{KW} for example), provided that the energy estimate \eqref{ES} holds true. From this standpoint, we will focus on in this paper
how to obtain the estimate \eqref{ES}.

\begin{Remark} Three remarks go as follows:
\begin{itemize}
 \item 
The solution in Theorem \ref{MT} is unique and the proof is a straightforward application of \eqref{ES}. We omit the proof of the
uniqueness here. 
\item  
Theorem \ref{MT} says that the global wellposedness of MHD with zero magnetic diffusivity strongly depends on the flow. In other words, to ensure the global wellposedness,
not only the perturbation is small, the coupling between the flow and the magnetic field needs to be very subtle (see Proposition \ref{p11} and \eqref{19}).
\item Theorem \ref{MT} and its strategy of proof can be extended to the three-dimensional case with a slight modification for the regularity of \textit{hybrid Besov spaces}. 
\end{itemize}
\end{Remark}

\bigskip\bigskip

\section{Besov Spaces}

Throughout this paper,  we use $C$ for a generic constant, and
denote  $a\le Cb$ by  $a\lesssim b$. The notation $a\thickapprox
b$ means that $a\lesssim b$ and $b\lesssim a$. Also we use
$(\alpha_q)_{q\in\mathbb{Z}}$ to denote a sequence such that
$\sum_{q\in\mathbb{Z}}\alpha_q\le 1$. $(f|g)$ denotes the inner
product of two functions $f, g$ in $L^2(\R^2)$. The standard
summation notation over the repeated index is adopted in this
paper.

The definition of the homogeneous Besov space is built on an
homogeneous Littlewood-Paley decomposition. First, we introduce a
function $\psi\in C^\infty(\R^2)$, supported in
$\mathcal{C}=\{\xi\in\R^2: \f{5}{6}\le|\xi|\le\f{12}{5}\}$ and
such that
$$\sum_{q\in\mathbb{Z}}\psi(2^{-q}\xi)=1\textrm{ if }\xi\neq 0.$$
Denoting $\mathcal{F}^{-1}\psi$ by $h$, we define the dyadic
blocks as follows:
$$\D_q f=\psi(2^{-q}D)f=2^{2q}\int_{\R^2}h(2^qy)f(x-y)dy,$$
and
$$S_q f=\sum_{p\le q-1}\D_pf.$$
The formal decomposition
\begin{equation}\label{21}
f=\sum_{q\in\mathbb{Z}}\D_qf
\end{equation}
is called homogeneous Littlewood-Paley decomposition in $\R^2$. Similarly, we use $\D_k^1$ to denote the homogeneous Littlewood-Paley decomposition in $\R$ in the direction of $x_1$.

For $s\in\R$ and $f\in \mathcal{S}'(\R^2)$, we denote
$$\|f\|_{\dot{B}^s_{p,r}}\overset{def}{=}\left(\sum_{q\in\mathbb{Z}}2^{sqr}\|\D_qf\|^r_{L^p}\right)^{\f{1}{r}}.$$
As $p=2$ and $r=1$, we denote $\|\cdot\|_{\dot{B}^s_{p,r}}$ by
$\|\cdot\|_{B^s}$.
\begin{Definition}
Let $s\in\R$, and $m=-\left[2-s\right]$. If $m<0$, we set
$$B^s=\left\{f\in \mathcal{S}'(\R^2)|\|f\|_{B^s}<\infty\textrm{
and }f=\sum_{q\in\mathbb{Z}}\D_qf\textrm{ in
}\mathcal{S}'(\R^2)\right\}.$$ If $m\ge 0$, we denote by
$\mathcal{P}_m$ the set of two variables polynomials of degree
$\le m$ and define
$$B^s=\left\{f\in \mathcal{S}'(\R^2)/\mathcal{P}_m|\|f\|_{B^s}<\infty\textrm{
and }f=\sum_{q\in\mathbb{Z}}\D_qf\textrm{ in
}\mathcal{S}'(\R^2)/\mathcal{P}_m\right\}.$$
\end{Definition}

Functions in $B^s$ has many good properties (see Proposition 2.5
in \cite{RD}):
\begin{Proposition}\label{p3}
The following properties hold:
\begin{itemize}
\item Derivation: $\|f\|_{B^s}\thickapprox \|\nabla f\|_{B^{s-1}}$;
\item Fractional derivation: let $\Gamma=\sqrt{-\D}$ and
$\sigma\in\R$; then the operator $\Gamma^\sigma$ is an isomorphism
from $B^s$ to $B^{s-\sigma}$;
\item Algebraic properties: for $s>0$, $B^s\cap L^\infty$ is an
algebra.
\end{itemize}
\end{Proposition}

To handle the degeneracy of the hyperbolicity, we further apply the Littlewood-Paley decomposition in $x_1$ to introduce a smaller Besov space $\hat{B}^s$ as
\begin{Definition}
Let $q, k\in \mathbb{Z}$ and $s\in\R$. We set
$$\|f\|_{\hat{B}^{s}}=\sum_{q,k\in\mathbb{Z}}2^{qs}\|\D_q\D_k^1f\|_{L^2}.$$
The Besov space $\hat{B}^{s}$ is defined by
$$\hat{B}^{s}=\left\{f\in\mathcal{S}'(\R^2)|\|f\|_{\hat{B}^{s}}<\infty\right\}.$$
Note that $\hat{B}^s\subset B^s$ for $s\in\R$ since $\ell_1\subset \ell_2$. In particular $\hat{B}^1\subset B^1\subset L^\infty$.
\end{Definition}
The product in Besov space $\hat{B}^s$ can be estimated by
\begin{Proposition}\label{p2}
For all $s, t\le 1$ such that $s+t>0$,
$$\|fg\|_{\hat{B}^{s+t-1}}\lesssim
\|f\|_{\hat{B}^{s}}\|g\|_{\hat{B}^t}.$$
\end{Proposition}
The proof of Proposition \ref{p2} will be postponed to the Appendix.

To deal with functions with different regularities for different frequencies as suggested by the spectral analysis in Section 4, it is
more effective to work in a \textit{hybrid Besov space} $\tilde{B}^{s,t}$. 
\begin{Definition}
Let $q, k\in \mathbb{Z}$ and $s,t\in\R$. We set
$$\|f\|_{\tilde{B}^{s,t}}=\sum_{k+1\ge 2q}2^{qs}\|\D_q\D_k^1f\|_{L^2}+\sum_{k+1< 2q}2^{(2q-k)t}\|\D_q\D_k^1f\|_{L^2}.$$
The \textit{hybrid Besov space} $\tilde{B}^{s,t}$ is defined by
$$\tilde{B}^{s,t}=\left\{f\in\mathcal{S}'(\R^2)|\|f\|_{\tilde{B}^{s,t}}<\infty\right\}.$$
\end{Definition}

Note that $\D_q\D_k^1f\neq 0$ only if $k\le q$. Thus the partial sum $\sum_{k+1\ge 2q}2^{qs}\|\D_q\D_k^1f\|_{L^2}$ actually only occurs at the low frequence $q\lesssim 1$.

\begin{Remark}
Three remarks go as follows.
\begin{itemize}
 \item The space $\tilde{B}^{s,t}$ is not empty since for any function $\phi\in C^\infty(\R^2)$ with support of its fourier transform 
in the union of balls $\{\xi\in\R^2|2|\xi_1|\ge |\xi|^2\}$, it holds $\mathcal{F}\phi\in \tilde{B}^{s,t}.$
\item $\tilde{B}^{0,1}$ is continuously embedded into $L^\infty$. Indeed, for any fixed $q,k\in\mathbb{Z}$, it holds
\begin{equation*}
\begin{split}
\|\D_q\D_k^1f\|_{L^\infty}&\le \|\mathcal{F}(\D_q\D_k^1f)\|_{L^1}\\
&\lesssim
\begin{cases}
\|\mathcal{F}(\D_q\D_k^1f)\|_{L^2}\textrm{meas}(\{2|\xi_1|\ge |\xi|^2\})\quad \textrm{if}\quad k+1\ge 2q;\\
\|\xi_1^{-1}|\xi|^2\mathcal{F}(\D_q\D_k^1f)\|_{L^2}\||\xi_1||\xi|^{-2}\|_{L^2(\{2|\xi_1|<|\xi|^2\}\cap \{2^q\le |\xi|\le 2^{q+1}\})} \quad \textrm{otherwise}
\end{cases}\\
&\lesssim \max\{2, 2^{2q-k}\}\|\D_q\D_k^1f\|_{L^2}\\
&\lesssim \|f\|_{\tilde{B}^{0,1}}.
\end{split}
\end{equation*}
Thus, $\sum_{q,k}\D_q\D_k^1f$ uniformly converges to $f$ in $L^\infty$.
\item If $0\le s\le t$, $\tilde{B}^{s,t}\subset \hat{B}^s\cap\hat{B}^t$, that is
\begin{equation}\label{25a}
\|\phi\|_{\hat{B}^s\cap\hat{B}^t}=\max\{\|\phi\|_{\hat{B}^s},\|\phi\|_{\hat{B}^t}\}\lesssim \|\phi\|_{\tilde{B}^{s,t}}.
\end{equation}
\end{itemize}
\end{Remark}
An estimate for the product in our hybird Besov space $\tilde{B}^{0,1}$ is needed.
\begin{Lemma}\label{le}
There holds true
$$\|fg\|_{\tilde{B}^{0,1}}\lesssim \|f\|_{\tilde{B}^{0,1}}\|g\|_{\hat{B}^1}.$$
\end{Lemma}
The proof of Lemma \ref{le} is similar to that of Proposition \ref{p2} (or Proposition 5.2 in \cite{RD}) and we omit it here.

In order to state our existence result, we introduce some
functional spaces and explain the notation. Let $T>0$,
$r\in[0,\infty]$ and $X$ be a Banach space. We denote by
$\mathcal{M}(0,T;X)$ the set of measurable functions on $(0,T)$
valued in $X$. For $f\in \mathcal{M}(0,T; X)$, we define
$$\|f\|_{L^r_T(X)}=\left(\int_0^T\|f(\tau)\|_X^rd\tau\right)^{\f{1}{r}}\textrm{
if }r<\infty,$$
$$\|f\|_{L^\infty_T(X)}=\sup \textrm{ess}_{\tau\in(0,T)}\|f(\tau)\|_X.$$
Denote $L^r(0,T;X)=\{f\in
\mathcal{M}(0,T;X)|\|f\|_{L^r_T(X)}<\infty\}$ If $T=\infty$, we
denote by $L^r(\R^+; X)$ and $\|f\|_{L^r(X)}$ the corresponding
spaces and norms. Also denote by $C([0,T],X)$ (or $C(\R^+,X)$) the
set of continuous X-valued functions on $[0,T]$ (resp. $\R^+$).

Throughout this paper, the following estimates for the convective
terms arising in the localized system is used several times.
\begin{Lemma}\label{cl}
Let $G$ be a smooth function away from the origin with the form $|\xi|^m\xi_1^n$. 
Then there hold true
\begin{equation}\label{25b}
 \begin{split}
&|(G(D)\D_q\D_k^1(e\cdot\nabla f)|G(D)\D_q\D_k^1 f)|\\
&\quad\le C\alpha_{q,k}
2^{nk+mq}
\|e\|_{\hat{B}^{2}}\|f\|_{\hat{B}^{0}}\left\|G(D)\D_q\D_k^1
f\right\|_{L^2},  
 \end{split}
\end{equation}
\begin{equation}\label{25}
\begin{split}
&|(G(D)\D_q\D_k^1(e\cdot\nabla f)|G(D)\D_q\D_k^1 f)|\\
&\quad\le C\alpha_{q,k}
2^{nk+mq}\min\{2^{-1}, 2^{k-2q}\}
\|e\|_{\hat{B}^{2}}\|f\|_{\tilde{B}^{0,1}}\left\|G(D)\D_q\D_k^1
f\right\|_{L^2},
\end{split}
\end{equation}
and
\begin{equation}\label{26}
\begin{split}
&\left|(G(D)\D_q\D_k^1(e\cdot\nabla f)|\D_q\D_k^1 g)+(\D_q\D_k^1(e\cdot\nabla
g)|G(D)\D_q\D_k^1 f)\right|\\&\le
C\alpha_{q,k}\|e\|_{\hat{B}^{2}}\Big(\left\|G(D)\D_q\D_k^1f\right\|_{L^2}\|g\|_{\hat{B}^0}+2^{nk+mq}\min\{2^{-1},2^{k-2q}\}
\|f\|_{\tilde{B}^{0,1}}\|\D_q\D_k^1g\|_{L^2}\Big),
\end{split}
\end{equation}
where $\sum_{q,k\in\mathbb{Z}}\alpha_{q,k}\le 1$.
\end{Lemma}
We postpone the proof of Lemma \ref{cl} to the Appendix.

\bigskip\bigskip
\section{Dissipation Estimates}

This section aims at the dissipation estimate for $(\u, \H)$.

\subsection{Dissipation of velocity}
In this subsection, we consider the following structure:
\begin{subequations}\label{a01}
 \begin{align}
&\partial_tu+u\cdot\nabla u-\D u-\partial_{x_1}v=L\label{a01a}\\
&\partial_tv+u\cdot\nabla v-\partial_{x_1}u=M.\label{a01b}
 \end{align}
\end{subequations}
where $u, v$ are vector-valued functions in $\R^2$, and $L,M$ are nonlinear terms of $(u,v)$.
\begin{Remark}
In system \eqref{a01}, instead the linear system \eqref{e1a}, we add convective terms because no matter how smooth functions $u,v$ are, convective terms $\u\cdot\nabla \H$ lose one derivative,
and hence induce difficulties to treat them as force terms.
\end{Remark}

Without convective terms, the solution of \eqref{a01} reads
\begin{equation*}
 \left(\begin{array}{cc}
u(t)\\ v(t)
 \end{array}\right)=e^{A(\partial)t} \left(\begin{array}{cc}
u_0\\ v_0
 \end{array}\right)+\int_0^te^{A(\partial)(t-\tau)}\left(\begin{array}{cc}
L(\tau)\\ M(\tau)
 \end{array}\right)d\tau
\end{equation*}
with 
\begin{equation*}
 A(\partial)=\left(\begin{array}{cc}
                    \D&\partial_{x_1}\\
                    \partial_{x_1}&0
                   \end{array}\right).
\end{equation*}
The operator $A(\partial)$ has a symbol
\begin{equation*}
 A(\xi)=\left(\begin{array}{cc}
                    -|\xi|^2&-i\xi_1\\
                    -i\xi_1&0
                   \end{array}\right),
\end{equation*}
and we now consider the eigenvalues of $A(\xi)$. By direct computations, the eigenvalues of $A(\xi)$ take the form of
$$\lambda_{\pm}=-\f12\Big(|\xi|^2\pm \sqrt{|\xi|^4-4\xi_1^2}\Big).$$

According to the frequence, one needs to consider two cases: $\f{2|\xi_1|}{|\xi|^2}\le 1$ and $\f{2|\xi_1|}{|\xi|^2}\ge 1$. As $\f{2|\xi_1|}{|\xi|^2}\ge 1$, the eigenvalues of $A(\xi)$ has
the form
$$\lambda_{\pm}=-\f{|\xi|^2}{2}\left(1\pm i\sqrt{\f{4\xi_1^2}{|\xi|^4}-1}\right),$$ and all eigenvectors have parabolic dissipations as $e^{-t|\xi|^2}$. As $\f{2|\xi_1|}{|\xi|^2}\le 1$, the eigenvalues of $A(\xi)$ take
the form
\begin{equation}\label{a2}
\lambda_{\pm}=-\f{|\xi|^2}{2}\left(1\pm \sqrt{1-\f{4\xi_1^2}{|\xi|^4}}\right),
\end{equation} and hence the eigenvector associated to $\lambda_+$ has a parabolic dissipation while the eigenvector associated to $\lambda_{-}$ has a damping effect
in the direction of $\xi_1$.

For the system \eqref{a01}, we have the following dissipation for $u$.
\begin{Proposition}\label{p00}
For solutions $(u,v)$ of \eqref{a01}, it holds true
\begin{equation}\label{a2a}
\begin{split}
&\|u\|_{L^\infty(\hat{B}^0)}+\|v\|_{L^\infty(\tilde{B}^{0,1})}+\|u\|_{L^1(\hat{B}^2)}\\
&\qquad+\int_0^\infty\left(\sum_{k+1\ge 2q}2^{2q}\|\D_q\D_k^1v\|_{L^2}+\sum_{k+1<2q}\|\D_q\D_k^1\partial_{x_1}v\|_{L^2}\right)dt\\
&\quad\lesssim \|u(0)\|_{\hat{B}^0}+\|v(0)\|_{\tilde{B}^{0,1}}+\|L\|_{L^1(\hat{B}^0)}+\|M\|_{L^1(\tilde{B}^{0,1})}\\
&\qquad+\|u\|_{L^1(\hat{B}^2)}(\|v\|_{L^\infty(\tilde{B}^{0,1})}+\|u\|_{L^\infty(\hat{B}^0)}).
\end{split}
\end{equation}
\end{Proposition}

\begin{proof} 
We divide the proof into two steps.

\texttt{Step 1: Energy Estimates.}

Applying the Littlewood-Paley decomposition $\D_q\D_k^1$ to \eqref{a01}, one has
\begin{subequations}\label{a1}
 \begin{align}
&\partial_t\D_q\D_k^1 u+\D_q\D_k^1(u\cdot\nabla u)-\D \D_q\D_k^1 u-\partial_{x_1}\D_q\D_k^1v=\D_q\D_k^1L\label{a1a}\\
&\partial_t\D_q\D_k^1v+\D_q\D_k^1(u\cdot\nabla v)-\partial_{x_1}\D_q\D_k^1u=\D_q\D_k^1M.\label{a1b}
 \end{align}
\end{subequations}
We consider two cases according to the frequence.

\textit{Case 1: $k+1\ge 2q$.} 
For $\iota>0$, define
$$f_{q,k}^2=\|\D_q\D_k^1 u\|_{L^2}^2+\|\D_q\D_k^1v\|_{L^2}^2-\iota2^{2q-2k+1}(\D_q\D_k^1u|\D_q\D_k^1\partial_{x_1}v).$$

Taking the $L^2$-product of \eqref{a1a} with $\D_q\D_k^1u$, we obtain
\begin{equation}\label{a4}
\begin{split}
 &\f12\f{d}{dt}\|\D_q\D_k^1u\|_{L^2}^2+(\D_q\D_k^1(u\cdot\nabla u)|\D_q\D_k^1u)+\|\Lambda \D_q\D_k^1u\|_{L^2}^2-(\partial_{x_1}\D_q\D_k^1v|\D_q\D_k^1u)\\
&\quad=(\D_q\D_k^1L|\D_q\D_k^1u).
\end{split}
\end{equation}
Taking the $L^2$-product of \eqref{a1b} with $\D_q\D_k^1v$, we obtain
\begin{equation}\label{a5}
\begin{split}
\f12\f{d}{dt}\|\D_q\D_k^1v\|_{L^2}^2+(\D_q\D_k^1(u\cdot\nabla v)|\D_q\D_k^1v)-(\partial_{x_1}\D_q\D_k^1u|\D_q\D_k^1v)=(\D_q\D_k^1M|\D_q\D_k^1v).
\end{split}
\end{equation}

For the cross term $(\D_q\D_k^1u|\D_q\D_k^1\partial_{x_1}v)$, applying $\partial_{x_1}$ to \eqref{a1b}, and then multiplying resulting equation and \eqref{a1a} 
by $\D_q\D_k^1\partial_{x_1}v$ and $\D_q\D_k^1u$ respectively, one obtains
\begin{equation}\label{a7}
\begin{split}
&\f{d}{dt}(\D_q\D_k^1u|\D_q\D_k^1\partial_{x_1}v)-\|\partial_{x_1}\D_q\D_k^1 v\|_{L^2}^2+\|\partial_{x_1}\D_q\D_k^1u\|_{L^2}^2+(\Lambda^2\D_q\D_k^1u|\D_q\D_k^1\partial_{x_1}v)\\
&\qquad+(\D_q\D_k^1(u\cdot\nabla u)|\D_q\D_k^1\partial_{x_1}v)+(\partial_{x_1}\D_q\D_k^1(u\cdot\nabla v)|\D_q\D_k^1u)\\
&\quad=(\D_q\D_k^1L|\D_q\D_k^1\partial_{x_1}v)+(\D_q\D_k^1\partial_{x_1}M|\D_q\D_k^1v).
\end{split}
\end{equation}

A linear combination of \eqref{a4}-\eqref{a7} gives
\begin{equation}\label{a8}
 \begin{split}
&\f12\f{d}{dt}f_{q,k}^2+(1-\iota)\|\Lambda\D_q\D_k^1 u\|_{L^2}^2+\iota\|\Lambda\D_q\D_k^1 v\|_{L^2}^2\\
&\quad-\iota 2^{2q-2k}(\Lambda^2\D_q\D_k^1u|\D_q\D_k^1\partial_{x_1}v)=\mathcal{X}_{q,k},
 \end{split}
\end{equation}
with
\begin{equation*}
 \begin{split}
\mathcal{X}_{q,k}&\overset{def}=(\D_q\D_k^1L|\D_q\D_k^1u)+(\D_q\D_k^1M|\D_q\D_k^1v)-\iota 2^{2q-2k}(\D_q\D_k^1L|\D_q\D_k^1\partial_{x_1}v)\\
&\quad-\iota 2^{2q-2k}(\D_q\D_k^1\partial_{x_1}M|\D_q\D_k^1v)+\iota 2^{2q-2k}(\D_q\D_k^1(u\cdot\nabla u)|\D_q\D_k^1\partial_{x_1}v)\\
&\quad+\iota 2^{2q-2k}(\partial_{x_1}\D_q\D_k^1(u\cdot\nabla v)|\D_q\D_k^1u)+(\D_q\D_k^1(u\cdot\nabla u)|\D_q\D_k^1u)\\
&\quad+(\D_q\D_k^1(u\cdot\nabla v)|\D_q\D_k^1v).
 \end{split}
\end{equation*}

Since $k+1\ge 2q$, Bernstein's inequality gives
\begin{equation}\label{a9}
2^{2q-2k}\|\partial_{x_1}\D_q\D_k^1\phi\|_{L^2}\lesssim \|\D_q\D_k^1\phi\|_{L^2},
\end{equation} 
and thus
$$f_{q,k}^2\approx \|\D_q\D_k^1u\|_{L^2}^2+\|\D_q\D_k^1v\|_{L^2}^2 $$if $\iota$ is chosen to be sufficiently small.

Since $k+1\ge 2q$, it follows that $|\xi|\lesssim 1$ or equivalently $q,k\lesssim 1$. Using Bernstein's inequality, one has 
\begin{equation*}
 \begin{split}
&(1-\iota)\|\Lambda\D_q\D_k^1 u\|_{L^2}^2+\iota\|\Lambda\D_q\D_k^1 v\|_{L^2}^2-\iota 2^{2q-2k}(\Lambda^2\D_q\D_k^1u|\D_q\D_k^1\partial_{x_1}v)\\
&\quad\approx \|\Lambda\D_q\D_k^1 u\|_{L^2}^2+\|\Lambda\D_q\D_k^1 v\|_{L^2}^2\ge C2^{2q}\Big(\|\D_q\D_k^1 u\|_{L^2}^2+\|\D_q\D_k^1 v\|_{L^2}^2\Big).
 \end{split}
\end{equation*}
Here we used
$$2^{2q-2k}|(\Lambda^2\D_q\D_k^1u|\D_q\D_k^1\partial_{x_1}v)|\lesssim \|\Lambda\D_q\D_k^1u\|_{L^2}\|\Lambda\D_q\D_k^1v\|_{L^2}.$$

Based on \eqref{a9}, one deduces from Lemma \ref{cl} that
$$|\mathcal{X}_{q,k}|\lesssim f_{q,k}\Big(\|\D_q\D_k^1L\|_{L^2}+\|\D_q\D_k^1M\|_{L^2}+\alpha_{q,k}\|u\|_{\hat{B}^2}(\|v\|_{\tilde{B}^{0,1}}+\|u\|_{\hat{B}^0})\|\Big),$$ and hence
it follows from \eqref{a8}, \eqref{a9} and Bernstein's inequality that there is a positive constant $\beta$ such that
\begin{equation}\label{a10}
\f{d}{dt}f_{q,k}+\beta 2^{2q} f_{q,k}\lesssim \|\D_q\D_k^1L\|_{L^2}+\|\D_q\D_k^1M\|_{L^2}+\alpha_{q,k}\|u\|_{\hat{B}^2}(\|v\|_{\tilde{B}^{0,1}}+\|u\|_{\hat{B}^0}).
\end{equation}
Summing \eqref{a10} over $q,k\in\mathbb{Z}$ and integrating over $t\in \R$, one obtains \eqref{a2a}.

\textit{Case 2: $k+1< 2q$.} 
Define
$$f_{q,k}^2=2\|\D_q\D_k^1\mathcal{R}_1^2u\|_{L^2}^2+\|\D_q\D_k^1\partial_{x_1}v\|_{L^2}^2+2(\D_q\D_k^1\mathcal{R}_1^2u, \D_q\D_k^1\partial_{x_1}v),$$
where $\mathcal{R}_1$ denotes the Riesz operator with the symbol $i\xi_1/|\xi|$.

Applying the operator $\mathcal{R}_1^2$ to \eqref{a1a}, and the taking the $L^2$ product of the resulting equation with $\D_q\D_k^1\mathcal{R}_1^2u$, one has
\begin{equation}\label{a10a}
\begin{split}
&\f12\f{d}{dt}\|\D_q\D_k^1\mathcal{R}_1^2u\|_{L^2}^2+(\mathcal{R}_1^2\D_q\D_k^1(u\cdot\nabla u)|\mathcal{R}_1^2\D_q\D_k^1u)+\|\mathcal{R}_1\D_q\D_k^1\partial_{x_1}u\|_{L^2}^2\\
&\quad-(\mathcal{R}_1^2\partial_{x_1}\D_q\D_k^1v|\mathcal{R}_1^2\D_q\D_k^1u)=(\mathcal{R}_1^2\D_q\D_k^1L|\mathcal{R}_1^2\D_q\D_k^1u).
\end{split}
\end{equation}

Applying the operator $\partial_{x_1}$ in \eqref{a1b}, and taking the $L^2$ product of the resulting equation with $\D_q\D_k^1\partial_{x_1}v$, one has
\begin{equation}\label{a6}
\begin{split}
&\f12\f{d}{dt}\|\D_q\D_k^1\partial_{x_1}v\|_{L^2}^2+(\partial_{x_1}\D_q\D_k^1(u\cdot\nabla v)|\D_q\D_k^1\partial_{x_1}v)-(\D_q\D_k^1\partial_{x_1}^2u|\D_q\D_k^1\partial_{x_1}v)\\
&\quad=(\D_q\D_k^1\partial_{x_1}M|\D_q\D_k^1\partial_{x_1}v).
\end{split}
\end{equation}

For the cross term $(\D_q\D_k^1\mathcal{R}_1^2u|\D_q\D_k^1\partial_{x_1}v)$, applying $\mathcal{R}_1^2$ and $\partial_{x_1}$ to \eqref{a1a} and \eqref{a1b}
respectively, and then multiplying resulting equations by $\D_q\D_k^1\partial_{x_1}v$ and $\D_q\D_k^1\mathcal{R}_1^2u$ respectively, one obtains
\begin{equation}\label{a7a}
\begin{split}
&\f{d}{dt}(\D_q\D_k^1\mathcal{R}_1^2u|\D_q\D_k^1\partial_{x_1}v)+\|\mathcal{R}_1\D_q\D_k^1\partial_{x_1} v\|_{L^2}^2-\|\mathcal{R}_1\D_q\D_k^1\partial_{x_1}u\|_{L^2}^2\\
&\qquad+(\mathcal{R}_1^2\D_q\D_k^1(u\cdot\nabla u)|\D_q\D_k^1\partial_{x_1}v)+(\partial_{x_1}\D_q\D_k^1(u\cdot\nabla v)|\mathcal{R}_1^2\D_q\D_k^1u)\\
&\qquad+(\partial_{x_1}^2\D_q\D_k^1u|\D_q\D_k^1\partial_{x_1}v)\\
&\quad=(\D_q\D_k^1\mathcal{R}_1^2L|\D_q\D_k^1\partial_{x_1}v)+(\D_q\D_k^1\partial_{x_1}M|\D_q\D_k^1\mathcal{R}_1^2v).
\end{split}
\end{equation}

A linear combination of \eqref{a10a}-\eqref{a7a} gives
\begin{equation}\label{a11}
 \begin{split}
&\f12\f{d}{dt}f_{q,k}^2+\|\mathcal{R}_1\D_q\D_k^1 \partial_{x_1}u\|_{L^2}^2+\|\mathcal{R}_1\D_q\D_k^1\partial_{x_1}v\|_{L^2}^2-2(\mathcal{R}_1^2\partial_{x_1}\D_q\D_k^1v|\mathcal{R}_1^2\D_q\D_k^1u)\\
&\quad=\mathcal{Y}_{q,k},
 \end{split}
\end{equation}
with
\begin{equation*}
 \begin{split}
\mathcal{Y}_{q,k}&\overset{def}=2(\mathcal{R}_1^2\D_q\D_k^1L|\mathcal{R}_1^2\D_q\D_k^1u)-2(\mathcal{R}_1^2\D_q\D_k^1(u\cdot\nabla u)|\mathcal{R}_1^2\D_q\D_k^1u)\\
&\quad+(\D_q\D_k^1\partial_{x_1}M|\D_q\D_k^1\partial_{x_1}v)-(\partial_{x_1}\D_q\D_k^1(u\cdot\nabla v)|\D_q\D_k^1\partial_{x_1}v)\\
&\quad+(\D_q\D_k^1\mathcal{R}_1^2L|\D_q\D_k^1\partial_{x_1}v)+(\D_q\D_k^1\partial_{x_1}M|\D_q\D_k^1\mathcal{R}_1^2v)\\
&\quad-(\mathcal{R}_1^2\D_q\D_k^1(u\cdot\nabla u)|\D_q\D_k^1\partial_{x_1}v)-(\partial_{x_1}\D_q\D_k^1(u\cdot\nabla v)|\mathcal{R}_1^2\D_q\D_k^1u)\\
 \end{split}
\end{equation*}

Observe that
\begin{equation}\label{a12a}
f_{q,k}^2\approx \|\D_q\D_k^1\mathcal{R}_1^2u\|_{L^2}^2+\|\D_q\D_k^1\partial_{x_1}v\|_{L^2}^2.
\end{equation} 

Since
\begin{equation*}
 \begin{split}
2|(\mathcal{R}_1^2\partial_{x_1}\D_q\D_k^1v|\mathcal{R}_1^2\D_q\D_k^1u)|\le \|\mathcal{R}_1\partial_{x_1}\D_q\D_k^1v\|_{L^2}\|\mathcal{R}_1\partial_{x_1}\D_q\D_k^1u\|_{L^2},
 \end{split}
\end{equation*}
there holds
\begin{equation*}
 \begin{split}
&\|\mathcal{R}_1\D_q\D_k^1 \partial_{x_1}u\|_{L^2}^2+\|\mathcal{R}_1\D_q\D_k^1\partial_{x_1}v\|_{L^2}^2-2(\mathcal{R}_1^2\partial_{x_1}\D_q\D_k^1v|\mathcal{R}_1^2\D_q\D_k^1u)\\
&\quad\approx \|\mathcal{R}_1\D_q\D_k^1 \partial_{x_1}u\|_{L^2}^2+\|\mathcal{R}_1\D_q\D_k^1\partial_{x_1}v\|_{L^2}^2\\
&\quad\ge C2^{2k-2q}\Big(\|\mathcal{R}^2_1\D_q\D_k^1 u\|_{L^2}^2+\|\D_q\D_k^1\partial_{x_1}v\|_{L^2}^2\Big).  
 \end{split}
\end{equation*}

On the other hand, one deduces from Lemma \ref{cl} that
\begin{equation}\label{a12}
|\mathcal{Y}_{q,k}|\lesssim f_{q,k}\Big(\|\D_q\D_k^1\mathcal{R}_1^2L\|_{L^2}+\|\D_q\D_k^1\partial_{x_1}M\|_{L^2}+\alpha_{q,k}2^{2k-2q}\|u\|_{\hat{B}^2}(\|v\|_{\tilde{B}^{0,1}}+\|u\|_{\hat{B}^0})\Big).
\end{equation}

Combining \eqref{a11} and \eqref{a12}, one has
\begin{equation*}
 \begin{split}
&\f{d}{dt}f_{q,k}^2+\kappa 2^{2k-2q} \Big(\|\D_q\D_k^1\mathcal{R}_1^2u\|_{L^2}^2+\|\D_q\D_k^1\partial_{x_1}v\|_{L^2}^2\Big)\\
&\quad\lesssim \Big(\|\D_q\D_k^1\mathcal{R}_1^2L\|_{L^2}+\|\D_q\D_k^1\partial_{x_1}M\|_{L^2}+\alpha_{q,k}2^{2k-2q}\|u\|_{\hat{B}^2}(\|u\|_{\hat{B}^0}+\|v\|_{\tilde{B}^{0,1}})\Big)f_{q,k}
 \end{split}
\end{equation*} for some positive constant $\kappa$.

According to the equivalence \eqref{a12a}, dividing the equation above by $f_{q,k}$ and then multiplying the resulting inequality by $2^{2q-2k}$, one obtains
\begin{equation}\label{a14}
 \begin{split}
&\f{d}{dt}\Big(\|\D_q\D_k^1u\|_{L^2}+2^{2q-k}\|\D_q\D_k^1v\|_{L^2}\Big)+\kappa \Big(\|\D_q\D_k^1\mathcal{R}_1^2u\|_{L^2}+\|\D_q\D_k^1\partial_{x_1}v\|_{L^2}\Big)\\
&\quad\lesssim \|\D_q\D_k^1L\|_{L^2}+2^{2q-k}\|\D_q\D_k^1M\|_{L^2}+\alpha_{q,k}\|u\|_{\hat{B}^2}(\|u\|_{\hat{B}^0}+\|v\|_{\tilde{B}^{0,1}}).
 \end{split}
\end{equation}
Summing \eqref{a14} over $q,k\in\mathbb{Z}$ and integrating over $t\in \R$, it follows
\begin{equation}\label{a15}
\begin{split}
&\sup_{0\le \tau<\infty}\sum_{k+1<2q}\Big(\|\D_q\D_k^1u(\tau)\|_{L^2}+2^{2q-k}\|\D_q\D_k^1v(\tau)\|_{L^2}\Big)+\sum_{k+1<2q}\|\D_q\D_k^1\partial_{x_1}v\|_{L^1(L^2)}\\
&\quad\lesssim \|u(0)\|_{\hat{B}^0}+\|v(0)\|_{\tilde{B}^{0,1}}+\|L\|_{L^1(\hat{B}^0)}+\|M\|_{L^1(\tilde{B}^{0,1})}\\
&\qquad+\|u\|_{L^1(\hat{B}^2)}(\|u\|_{L^\infty(\hat{B}^0)}+\|v\|_{L^\infty(\tilde{B}^{0,1})}).
\end{split}
\end{equation}

\texttt{Step 2: Smoothing Effect.} We are going to use the parabolicity of $u$ to improve the dissipation of $u$. According to \eqref{a10}, it is only left to handle the case $k+1< 2q$.
Going back to the equation \eqref{a01a}, and regarding now the term $\partial_{x_1}v$ as an external term, one has
\begin{equation*}
\begin{split}
&\f12\f{d}{dt}\|\D_q\D_k^1u\|_{L^2}^2+2^{2q}\|\D_q\D_k^1u\|_{L^2}+(\D_q\D_k^1(u\cdot\nabla u)|\D_q\D_k^1u)\\
&\quad\lesssim \Big(\|\D_q\D_k^1\partial_{x_1}v\|_{L^2}+\|\D_q\D_k^1L\|_{L^2}\Big)\|\D_q\D_k^1u\|_{L^2}, 
\end{split}
\end{equation*}
and hence
\begin{equation}\label{a16}
\begin{split}
&\sum_{k+1<2q}2^{2q}\|\D_q\D_k^1u\|_{L^1(L^2)}\\
&\quad\lesssim \|u(0)\|_{\hat{B}^0}+\|L\|_{L^1(\hat{B}^0)}+\sum_{k+1<2q}\|\D_q\D_k^1\partial_{x_1}v\|_{L^1(L^2)}+\|u\|_{\hat{B}^2}\|u\|_{\hat{B}^0}.
\end{split}
\end{equation}
The desired estimate \eqref{a2a} follows from \eqref{a10}, \eqref{a15} and \eqref{a16}.
\end{proof}

Taking divergence in the momentum equation of \eqref{e2} yields
\begin{equation}\label{a16b}
\D\left(P+\f12|{\bf B}|^2\right)=\Dv\left(\H\cdot\nabla\H-\u\cdot\nabla\u\right),
\end{equation}
which means that $P+\f12|{\bf B}|^2$ is essentially a quadratic term in $(\u,\H)$ and hence can be regarded as an external term in \eqref{11} to get
\begin{subequations}\label{a16a}
\begin{align}
&\partial_t\u+\u\cdot\nabla\u-\D\u-\partial_{x_1}\H=-\nabla\left(P+\f12|{\bf B}|^2\right)+\mathcal{L};\\
&\partial_t\H+\u\cdot\nabla\H-\partial_{x_1}\u=\mathcal{M}
\end{align}
\end{subequations}
with
$$\mathcal{L}\overset{def}=\H\cdot\nabla\H\quad\textrm{and}\quad \mathcal{M}\overset{def}=\H\cdot\nabla\u.$$
From \eqref{a16b}, one has
$$\left\|\nabla \left(P+\f12|{\bf B}|^2\right)\right\|_{\hat{B}^0}\lesssim \|\mathcal{L}\|_{\hat{B}^0}.$$
Applying Proposition \ref{p00} to the system \eqref{a16a} yields
\begin{Corollary}\label{c1}
For solutions $(\u, \H)$ of \eqref{e2}, there holds true
\begin{equation*}
\begin{split}
&\|\u\|_{L^\infty(\hat{B}^0)}+\|\H\|_{L^\infty(\tilde{B}^{0,1})}+\|\u\|_{L^1(\hat{B}^2)}\\
&\qquad+\int_0^\infty\left(\sum_{k+1\ge 2q}2^{2q}\|\D_q\D_k^1\H\|_{L^2}+\sum_{k+1<2q}\|\D_q\D_k^1\partial_{x_1}\H\|_{L^2}\right)dt\\
&\quad\lesssim \|\u_0\|_{\hat{B}^0}+\|\H_0\|_{\tilde{B}^{0,1}}+\|\mathcal{L}\|_{L^1(\hat{B}^0)}+\|\mathcal{M}\|_{L^1(\tilde{B}^{0,1})}\\
&\qquad+\|\u\|_{L^1(\hat{B}^2)}(\|\u\|_{L^\infty(\hat{B}^0)}+\|\H\|_{L^\infty(\tilde{B}^{0,1})}).
\end{split}
\end{equation*}
\end{Corollary}

\subsection{Dissipation of $\H$}
Following line by line as \eqref{17} and \eqref{18}, using $\Dv\H=0$, one has
\begin{equation}\label{322}
 \begin{split}
&\|\Lambda\D_q\D_k^1\H\|_{L^2}^2+\left\|\Lambda\D_q\D_k^1\H_1\right\|_{L^2}^2\\
&\quad=-(\D_q\D_k^1\partial_{x_1}\H|\D_q\D_k^1\Dv \mathcal{A})
-\sum_{j=1}^2\left(\D_q\D_k^1\f{\partial\det\mathcal{A}}{\partial x_j}|\D_q\D_k^1\f{\partial^2 \alpha_2}{\partial x_2\partial x_j}\right).
 \end{split}
\end{equation}
Since $\mathcal{A}_{22}=\f{\partial \alpha_2}{\partial x_2}=\H_1$, it follows from \eqref{322} that
\begin{equation}\label{323}
 \begin{split}
&\|\Lambda\D_q\D_k^1\H\|_{L^2}^2+\left\|\Lambda\D_q\D_k^1\H_1\right\|_{L^2}^2\\
&\quad\lesssim \|\D_q\D_k^1\partial_{x_1}\H\|_{L^2}\|\D_q\D_k^1\Dv \mathcal{A}\|_{L^2}+\left\|\Lambda\D_q\D_k^1\det\mathcal{A}\right\|^2_{L^2}.
 \end{split}
\end{equation}
We claim now that
\begin{Lemma}\label{l2}
Assume that $\sup_{t\in[0,T]}\|\mathcal{A}(t)\|_{\hat{B}^{1}}\le \epsilon$ for $0\le T\le\infty$ and sufficiently small $\epsilon$. Then for solutions $(\u,\H)$ of \eqref{e2}, there holds
\begin{equation}\label{327}
 \begin{split}
\left\|\H\right\|_{L_T^2(\hat{B}^1)}^2&\lesssim \|\H\|_{L^\infty_T(\tilde{B}^{0,1})}\int_0^T\left(\sum_{k+1\ge 2q}2^{2q}\|\D_q\D_k^1\H\|_{L^2}\right)dt\\
&\qquad+\|\mathcal{A}\|_{L_T^\infty(\hat{B}^{1})}\int_0^T\left(\sum_{k+1<2q}\|\D_q\D_k^1\partial_{x_1}\H\|_{L^2}\right)dt.
 \end{split}
\end{equation}
\end{Lemma}
\begin{proof}
The definition of $\hat{B}^1$ gives
\begin{equation}\label{324a}
\|\H\|_{L^2_T(\hat{B}^1)}^2+\|\H_1\|_{L_T^2(\hat{B}^1)}^2=\int_0^T\left(\sum_{q,k\in\mathbb{Z}}2^q\Big(\|\D_q\D_k^1\H\|_{L^2}+\|\D_q\D_k^1\H_1\|_{L^2}\Big)\right)^2dt.
\end{equation}
We consider two cases according to the decomposition of frequencies.

\texttt{Case 1: $k+1\ge 2q$.} Holder's inequality for series implies
\begin{equation}\label{324c}
 \begin{split}
&\int_0^T\left(\sum_{k+1\ge 2q}2^q\|\D_q\D_k^1\H\|_{L^2}\right)^2dt=\int_0^T\left(\sum_{k+1\ge 2q}2^{q}\|\D_q\D_k^1\H\|_{L^2}^{1/2}\|\D_q\D_k^1\H\|_{L^2}^{1/2}\right)^2dt\\
&\quad\lesssim \int_0^T\left(\sum_{k+1\ge 2q}2^{2q}\|\D_q\D_k^1\H\|_{L^2}\right)\left(\sum_{k+1\ge 2q}\|\D_q\D_k^1\H\|_{L^2}\right)dt\\
&\quad\lesssim \|\H\|_{L^\infty_T(\tilde{B}^{0,1})}\int_0^T\left(\sum_{k+1\ge 2q}2^{2q}\|\D_q\D_k^1\H\|_{L^2}\right)dt.
 \end{split}
\end{equation}

\texttt{Case 2: $k+1<2q$.}
Using \eqref{323}, one has
\begin{equation}\label{324}
 \begin{split}
&\int_0^T\left(\sum_{k+1< 2q}2^q\Big(\|\D_q\D_k^1\H\|_{L^2}+\|\D_q\D_k^1\H_1\|_{L^2}\Big)\right)^2dt\\
&\qquad\lesssim \int_0^T\left(\sum_{k+1< 2q}2^{q/2}\|\D_q\D_k^1\partial_{x_1}\H\|_{L^2}^{1/2}\|\D_q\D_k^1\mathcal{A}\|_{L^2}^{1/2}\right)^2dt\\
&\quad\qquad+\int_0^T\left(\sum_{k+1< 2q}2^q\|\D_q\D_k^1\det\mathcal{A}\|_{L^2}\right)^2dt\\
&\qquad= \int_0^T\left(\sum_{k+1< 2q}2^{q/2}\|\D_q\D_k^1\partial_{x_1}\H\|_{L^2}^{1/2}\|\D_q\D_k^1\mathcal{A}\|_{L^2}^{1/2}\right)^2dt+\|\det\mathcal{A}\|_{L^2_T(\hat{B}^1)}^2.
 \end{split}
\end{equation}
For the first term in the right hand side, one has
\begin{equation}\label{324b}
 \begin{split}
 &\int_0^T\left(\sum_{k+1< 2q}2^{q/2}\|\D_q\D_k^1\partial_{x_1}\H\|_{L^2}^{1/2}\|\D_q\D_k^1\mathcal{A}\|_{L^2}^{1/2}\right)^2dt\\
&\quad \le\int_0^T\left(\sum_{k+1< 2q}\|\D_q\D_k^1\partial_{x_1}\H\|_{L^2}\right)\left(\sum_{k+1< 2q}2^q\|\D_q\D_k^1\mathcal{A}\|_{L^2}\right)dt\\
&\quad\lesssim \|\mathcal{A}\|_{L_T^\infty(\hat{B}^{1})}\int_0^T\sum_{k+1< 2q}\|\D_q\D_k^1\partial_{x_1}\H\|_{L^2}dt.
 \end{split}
\end{equation}
Substituting \eqref{324c}-\eqref{324b} to \eqref{324a} gives
\begin{equation}\label{325}
\begin{split}
&\|\H\|_{L^2(B^1)}^2+\|\H_1\|_{L_T^2(B^1)}^2-C\|\det\mathcal{A}\|_{L^2_T(B^1)}^2\\
&\quad\lesssim \|\H\|_{L^\infty_T(\tilde{B}^{0,1})}\int_0^T\left(\sum_{k+1\ge 2q}2^{2q}\|\D_q\D_k^1\H\|_{L^2}\right)dt\\
&\qquad+\|\mathcal{A}\|_{L_T^\infty(\hat{B}^{1})}\int_0^T\left(\sum_{k+1<2q}\|\D_q\D_k^1\partial_{x_1}\H\|_{L^2}\right)dt. 
\end{split}
\end{equation}

The identity \eqref{16a} implies
\begin{equation*}
 \begin{split}
\det\mathcal{A}=\mathcal{A}_{11}\mathcal{A}_{22}-\mathcal{A}_{12}\mathcal{A}_{21}=\mathcal{A}_{11}\H_1+\mathcal{A}_{12}\H_2,
 \end{split}
\end{equation*}
and hence Proposition \ref{p2} yields
\begin{equation}\label{326}
\begin{split}
 \|\det\mathcal{A}\|_{L^2_T(\hat{B}^1)}\lesssim \|\mathcal{A}\|_{L_T^\infty(\hat{B}^1)}\|\H\|_{L^2_T(\hat{B}^1)}.
\end{split}
\end{equation}
The desired estimate \eqref{327} then follows from \eqref{325}, \eqref{326}, and the assumption \\$\sup_{t\in[0,T]}\|\mathcal{A}(t)\|_{\hat{B}^{1}}\le \epsilon$ with $C\epsilon^2\le 1/2$.
\end{proof}

\bigskip\bigskip

\section{Proof of Theorem \ref{MT}}

This section aims at the proof of Theorem \ref{MT}, and we focus on the existence part of \ref{MT}. In order to show that a local solution can be extended to be a global one, we only need to prove 
the uniform estimate \eqref{ES}. For this purpose, we denote
$$X(t)=\|\mathcal{A}\|_{L^\infty_t(\hat{B}^{1})}+\|(\u,\H)\|_{\mathfrak{B}_t^1},$$
and we are going to show
\begin{equation}\tag{$\mathfrak{G}$}
X(t)\le C\Big(X(0)+X^2(t)\Big).
\end{equation}
Once ($\mathfrak{G}$) was shown, the existence part of Theorem \ref{MT} is done since by the continuity of $X(t)$ and the smallness of the initial data, there exists a constant $\Gamma$ such that
$$X(t)\le \Gamma X(0),$$ and hence local solutions can be extended.

The rest of this section is devoted to the proof of ($\mathfrak{G}$). To begin with, we establish the estimates for $\|\mathcal{A}\|_{L^\infty_t(\hat{B}^{1})}$.
\begin{Lemma}\label{l6a}
\begin{equation*}
\|\mathcal{A}\|_{L^\infty_t(\hat{B}^{1})}\lesssim X(0)+X(t)^2.
\end{equation*}
\end{Lemma}
\begin{proof} 
From \eqref{17d}, the function $\mathcal{A}$ satisfies a tranport equation
\begin{equation*}
 \partial_t\mathcal{A}+\u\cdot\nabla\mathcal{A}+\nabla\u=\mathcal{N}
\end{equation*}
with $\mathcal{N}=-\mathcal{A}\nabla\u$,
and hence according to Lemma \ref{cl}, there holds 
\begin{equation}\label{65}
 \begin{split}
\f12\f{d}{dt}\|\Lambda\D_q\D_k^1\mathcal{A}\|_{L^2}^2&=-(\Lambda\D_q\D_k^1\nabla\u|\Lambda\D_q\D_k^1\mathcal{A})-(\Lambda\D_q\D_k^1(\u\cdot\nabla\mathcal{A})|\Lambda\D_q\D_k^1\mathcal{A})\\
&\quad+(\Lambda\D_q\D_k^1\mathcal{N}|\Lambda\D_q\D_k^1\mathcal{A})\\
&\lesssim \|\Lambda\D_q\D_k^1\mathcal{A}\|_{L^2}\Big(\|\Lambda\D_q\D_k^1\nabla\u\|_{L^2}+\|\Lambda\D_q\D_k^1\mathcal{N}\|_{L^2}\\
&\quad+\alpha_{q,k}\|\u\|_{\hat{B}^2}\|\mathcal{A}\|_{\hat{B}^{1}}\Big).
\end{split}
\end{equation}
Summing \eqref{65} over $q,k\in\mathbb{Z}$ and integrating over $t$, one obtains
\begin{equation}\label{65b}
 \begin{split}
\sup_{t\ge 0}\|\mathcal{A}(t)\|_{\hat{B}^1}&\lesssim \|\mathcal{A}_0\|_{\hat{B}^1}
+\|\nabla\u\|_{L^1(\hat{B}^{1})}+\|\mathcal{N}\|_{L^1(\hat{B}^{1})}+\|\u\|_{L^1(\hat{B}^2)}\|\mathcal{A}\|_{L^\infty(\hat{B}^{1})}.
 \end{split}
\end{equation}
Using Lemma \ref{le}, one has
\begin{equation*}
 \begin{split}
\|\mathcal{N}\|_{L^1(\hat{B}^{1})}&\lesssim \|\mathcal{A}\|_{L^\infty(\hat{B}^{1})}\|\nabla\u\|_{L^1(\hat{B}^1)}\lesssim \|\mathcal{A}\|_{L^\infty(\hat{B}^{1})}\|\u\|_{L^1(\hat{B}^2)}\lesssim X(t)^2.
 \end{split}
\end{equation*}
Substituting this into \eqref{65b}, one has
\begin{equation*}
 \begin{split}
\|\mathcal{A}(t)\|_{\hat{B}^{1}}&\lesssim \|\mathcal{A}_0\|_{\hat{B}^{1}}+X(t)^2.
 \end{split}
\end{equation*}
Taking the $L^\infty$ norm over time gives the desired estimate.
\end{proof}

Next we turn to the estimate for $\|(\u,\H)\|_{\mathfrak{B}_t^1}$.
\begin{Lemma}\label{l6}
\begin{equation*}
 \begin{split}
 \|(\u,\H)\|_{\mathfrak{B}_t^1}\lesssim X(0)+X(t)^2. 
 \end{split}
\end{equation*}
\end{Lemma}
\begin{proof}
In view of Corollary \ref{c1} and Lemma \ref{l6a}, one has
\begin{equation}\label{61}
\|(\u,\H)\|_{\mathfrak{B}_t^1}\lesssim X(0)+\|\mathcal{L}\|_{L^1(\hat{B}^0)}+\|\mathcal{M}\|_{L^1(\tilde{B}^{0,1})}+X(t)^2.
\end{equation}

\texttt{Estimate of $\|\mathcal{L}\|_{L^1(B^0)}$.} Easily, one deduces from Corollary \ref{c1}
\begin{equation}\label{62}
\|\H\cdot\nabla\H\|_{L^1(\hat{B}^0)}\lesssim \|\H\|_{L^2(\hat{B}^1)}\|\nabla\H\|_{L^2(\hat{B}^0)}\lesssim \|\H\|_{L^2(\hat{B}^1)}^2\lesssim X(t)^2.
\end{equation}

\texttt{Estimate of $\|\mathcal{M}\|_{L^1(\tilde{B}^{0,1})}$.} This estimate is quite straightforward since Lemma \ref{le} implies
\begin{equation}\label{64}
 \begin{split}
\|\H\cdot\nabla\u\|_{L^1(\tilde{B}^{0,1})}&\lesssim \|\H\|_{L^\infty(\tilde{B}^{0,1})}\|\nabla\u\|_{L^1(\hat{B}^1)}\\
&\lesssim \|\H\|_{L^\infty(\tilde{B}^{0,1})}\|\u\|_{L^1(\hat{B}^2)}\lesssim X(t)^2.
 \end{split}
\end{equation}

Summaring \eqref{61}-\eqref{64} together yields the desired estimate.
\end{proof}

\bigskip\bigskip

\section{Appendix: Proofs of Product Laws}

This appendix is devoted to the proof of Proposition \ref{p2} and Lemma \ref{cl}. The ideas to prove them are not new (see for example \cite{RD}), but the proof requires paradifferential calculus. 
The isotropic para-differential decomposition of Bony form in $\R^2$ can be stated as follows: let $f,g\in\mathcal{S}'(\R^2)$,
$$fg=T(f,g)+\bar{T}(f,g)+R(f,g)$$ with $\bar{T}(f,g)=T(g,f)$ and
$$T(f,g)\overset{def}=\sum_{j\in\mathbb{Z}}S_{j-1}f\D_jg,\quad R(f,g)\overset{def}=\sum_{j\in\mathbb{Z}}\D_jf\tilde{\D}_jg,\quad \tilde{\D}_jg\overset{def}=\sum_{l=j-1}^{j+1}\D_lg.$$
We use $T^1 $, $\bar{T}^1$ and $R^1$ to denote the isotropic para-differential decomposition of Bony form for $\R$ in the direction of $x_1$ respectively.

We first give the proof of Proposition \ref{p2}.
\begin{proof}[Proof of Proposition \ref{p2}]
By Bony's decomposition, one has
\begin{equation}\label{71}
 fg=\Big(TT^1+T\bar{T}^1+TR^1+\bar{T}T^1+\bar{T}\bar{T}^1+\bar{T}R^1+RT^1+R\bar{T}^1+RR^1\Big)(f,g).
\end{equation}
We focus on estimates for typical terms such as $TR^1$ and $RR^1$. Other terms can be estimated similarly.

\texttt{Estimate of $TR^1$.} Since
\begin{equation*}
 \begin{split}
 \|S_{q'-1}\D_{k'}^1f\|_{L^\infty}\lesssim \sum_{p\le q'-2}2^p\|\D_{p}\D_{k'}^1f\|_{L^2}\lesssim 2^{q'(1-s)}\|f\|_{\hat{B}^s}.
 \end{split}
\end{equation*}
Since $\mathcal{F}(\D_{q'} f\tilde{\D}_{q'}g)$ is contained in $\beta\{|\xi|\le 2^{q'}\}$ for some $0<\beta$, the inequality above entails
\begin{equation*}
 \begin{split}
\|\D_q\D_k^1(TR^1(f,g))\|_{L^2}&\lesssim \sum_{\substack{|q'-q|\le 3\\k'\ge k-2}}\|S_{q'-1}\D_{k'}^1f\|_{L^\infty}\|\D_{q'}\tilde{\D}_{k'}^1g\|_{L^2}\\
&\lesssim \|f\|_{\hat{B}^s}\sum_{\substack{|q'-q|\le 3\\k'\ge k-2}}2^{q'(1-s)}\|\D_{q'}\tilde{\D}_{k'}^1g\|_{L^2},
 \end{split}
\end{equation*}
and hence combining Holder and convolution inequalities for series gives
$$\|TR^1(f,g)\|_{\hat{B}^{s+t-1}}\lesssim \|f\|_{\hat{B}^s}\|g\|_{\hat{B}^t}.$$

\texttt{Estimate of $RR^1$.}
\begin{equation*}
 \begin{split}
 \|\D_q\D_k^1(RR^1(f,g))\|_{L^2}&\lesssim \sum_{\substack{q'\ge q-2\\k'\ge k-2}}\|\D_{q'}\D_{k'}^1f\|_{L^\infty}\|\tilde{\D}_{q'}\tilde{\D}_{k'}^1g\|_{L^2}\\
&\lesssim 2^q\sum_{\substack{q'\ge q-2\\k'\ge k-2}}\|\D_{q'}\D_{k'}^1f\|_{L^2}\|\tilde{\D}_{q'}\tilde{\D}_{k'}^1g\|_{L^2}.
 \end{split}
\end{equation*}
Thus we have
\begin{equation*}
 \begin{split}
2^{q(s+t-1)}\|\D_q\D_k^1(RR^1(f,g))\|_{L^2}\lesssim \sum_{\substack{q'\ge q-2\\k'\ge k-2}}2^{(q-q')(s+t)}\alpha_{q',k'}\|f\|_{\hat{B}^s}\|g\|_{\hat{B}^t},
 \end{split}
\end{equation*}
and the convolution inequalities for series over $q$ yields
$$\|RR^1(f,g)\|_{\hat{B}^{s+t-1}}\lesssim \|f\|_{\hat{B}^s}\|g\|_{\hat{B}^t}$$
since $s+t>0.$
\end{proof}

Next we turn to the proof of Lemma \ref{cl}. The key idea is to apply the integration by parts to convert the derivative on $f$ or $g$ to the derivative on $\u$.
\begin{proof}[Proof of Lemma \ref{cl}]
To prove \eqref{25a} and \eqref{25}, one can use \eqref{71} to decompose the product $e\cdot\nabla f$ into nine pieces; and then estimate term by term. For illustration, let us consider
\begin{equation}\label{72}
 \begin{split}
&\sum_{\substack{|q'-q|\le 3\\k'\ge k-2}}\Big(G(D)\D_q\D_k^1(S_{q'-1}\D_{k'}^1e_j\D_{q'}\tilde{\D}_{k'}^1\partial_{x_j} f)|G(D)\D_q\D_k^1 f\Big),
 \end{split}
\end{equation}
and the worst term above is 
\begin{equation}\label{73}
\Big(S_{q-1}\D_k^1e_jG(D)\D_{q}\D_{k}^1\partial_{x_j} f|G(D)\D_q\D_k^1 f\Big)
\end{equation}
since the difference between \eqref{72} and \eqref{73} can be estimate with the aid of the first order Taylor's formula. 

For \eqref{73}, integration by parts gives
\begin{equation}\label{75}
 \begin{split}
&\left|\Big(S_{q-1}\D_k^1e_jG(D)\D_{q}\D_{k}^1\partial_{x_j} f|G(D)\D_q\D_k^1 f\Big)\right|\\
&\quad=\left|\Big(S_{q-1}\D_k^1\Dv e G(D)\D_{q}\D_{k}^1 f|G(D)\D_q\D_k^1 f\Big)\right|\\
&\quad\lesssim \|S_{q-1}\D_k^1\Dv e\|_{L^\infty} \|G(D)\D_q\D_k^1 f\|_{L^2}^2\\
&\quad\lesssim
\begin{cases}
\alpha_{q,k}2^{qm+nk}\|e\|_{\hat{B}^2}\|f\|_{\hat{B}^0}\|G(D)\D_q\D_k^1g\|_{L^2}\\
\alpha_{q,k}2^{qm+nk}\min\{2^{-1}, 2^{k-2q}\}\|e\|_{\hat{B}^2}\|f\|_{\tilde{B}^{0,1}}\|G(D)\D_q\D_k^1g\|_{L^2}.
\end{cases}
 \end{split}
\end{equation}

To prove \eqref{26}, we again focus on the piece
\begin{equation*}
 \begin{split}
&\sum_{\substack{|q'-q|\le 3\\k'\ge k-2}}\Big[\Big(G(D)\D_q\D_k^1(S_{q'-1}\D_{k'}^1e_j\D_{q'}\tilde{\D}_{k'}^1\partial_{x_j} f)|\D_q\D_k^1 g\Big)\\
&\quad+\Big(\D_q\D_k^1(S_{q'-1}\D_{k'}^1e_j\D_{q'}\tilde{\D}_{k'}^1\partial_{x_j} g)|G(D)\D_q\D_k^1 f\Big)\Big],
 \end{split}
\end{equation*}
and the worst term above is
\begin{equation*}
 \begin{split}
&\Big(S_{q-1}\D_{k}^1e_jG(D)\D_{q}\D_{k}^1\partial_{x_j} f|\D_q\D_k^1 g\Big)\\
&\quad+\Big(S_{q-1}\D_{k}^1e_j\D_{q}\D_{k}^1\partial_{x_j} g|G(D)\D_q\D_k^1 f\Big),
 \end{split}
\end{equation*}
which equals, using integration by parts
\begin{equation}\label{76}
 \begin{split}
-\Big(S_{q-1}\D_{k}^1\Dv eG(D)\D_{q}\D_{k}^1 f|\D_q\D_k^1 g\Big).
 \end{split}
\end{equation}
Similarly as \eqref{75}, we can estimate \eqref{76} to obtain \eqref{26}.

\end{proof}


\bigskip\bigskip


\end{document}